\documentclass{amsart}
\usepackage{amssymb,amsmath,amsfonts}
\usepackage{amsrefs}
\usepackage[toc,page]{appendix}

\begin{document}
\renewcommand{\theequation}{\thesection.\arabic{equation}}
\newtheorem{defi}{Definition}
\newtheorem{lem}{Lemma}
\newtheorem{prop}{Proposition}
\newtheorem{cor}{Corollary}
\newtheorem*{re}{Remark}
\newtheorem{thm}{Theorem}
\newtheorem{conj}{Conjecture}
\def\R{{\mathbb R}}
\def\N{{\mathbb N}}
\def\Z{{\mathbb Z}}
\def\IQ{{\mathbb{Q}}}
\def\IC{{\mathbb{C}}}
\makeatletter
\def\imod#1{\allowbreak\mkern5mu({\operator@font mod}\,\,#1)}
\makeatother
\renewcommand{\Re}{\mathop{\rm Re}}
\renewcommand{\Im}{\mathop{\rm Im}}
\newcommand{\chip}{\chi_{\mathbb{P}}}
\newcommand{\lcm}{\text{lcm}}
\renewcommand{\i}{\text{ for each $i$}}
\numberwithin{equation}{section}

\title{A generalization of Bombieri-Vinogradov theorem and its application}
\author{Peter Cho-Ho Lam}
\address{\noindent Department of Mathematics, Simon Fraser University,
Burnaby B.C., Canada\\  V5A 1S6.
Email:  {\tt chohol@sfu.ca}}
\date{\today}


\begin{abstract} In this paper we establish a generalization of Bombieri-Vinogradov theorem for primes represented by a fixed positive definite binary quadratic form. Then we apply this theorem to generalize a result of Vatwani on bounded gap between Gaussian primes.
\end{abstract}

\maketitle

\section{Introduction}
Let $\pi(x)$ be the number of primes less than $x$ and $\pi(x; q, a)$ be the number of primes $p$ less than $x$ that satisfies $p\equiv a\imod q$. It is widely believed that primes are equdistributed over arithmetic progressions; that is
\begin{equation}\label{equi}
\pi(x; q, a)\sim\frac{\pi(x)}{\phi(q)}\hspace{5mm}\text{as $x\rightarrow\infty$,}
\end{equation}
at least when $q$ is not too close to $x$. For example, the celebrated Siegel-Walfisz theorem states that \eqref{equi} is true when $q\le (\log x)^N$ where $N$ is a fixed positive constant. Assuming GRH, one can prove that \eqref{equi} holds for all $q\le x^{1/2-\epsilon}$ for any $\epsilon>0$. We are very far from proving this unconditionally, but the Bombieri-Vinogradov theorem asserts that this is true on average: for any $0<\theta<1/2$ and $A>0$, we have
\[
\sum_{q<x^\theta}\max_{\substack{a\imod q\\(a, q)=1}}\max_{y\le x}\bigg|\pi(y; q, a)-\frac{\pi(y)}{\phi(q)}\bigg|\ll_A x(\log x)^{-A}.
\]
Such theorem is very crucial to many studies of primes that involve the use of sieve methods.

Some recent  applications of Bombieri-Vinogradov theorem include the works of Maynard \cite{May} and Tao on bounded gaps between primes. This stimulated many research on bounded gaps between primes in special subsets; for example, Thorner \cite{Th} proved that if $Q(x, y)$ is a primitive positive definite binary quadratic form, then an analogous bounded gap result is true for the set of primes represented by $Q$: there is a constant $c(Q)>0$ such that there exist infinitely many primes $p_1, p_2$ such that $|p_1-p_2|\le c(Q)$ and both $p_1, p_2$ can be represented by $Q$. Later Vatwani \cite{Vat} generalized his result for the case $Q(x, y)=x^2+y^2$ which allows the first argument of $Q$ to be fixed.
\begin{thm}[Vatwani]
There are infinitely many primes of the form $p_1=a^2+b^2$ and $p_2=a^2+(b+h)^2$, with $a, b, h\in\Z$, such that $0<|h|\le 246$.
\end{thm}

Vatwani mentioned that the same result should hold for all imaginary quadratic field of class number one (e.g. for $Q(x, y)=x^2+163y^2$) with minor changes in the proof. This is because the corresponding ring of integers would still enjoy unique factorization and are therefore easier to work with. In this paper, we will prove a generalization of Bombieri-Vinogradov theorem that allows us to generalize Theorem 1 to most primitive positive definite binary quadratic forms:
\begin{prop}
\label{BVbqf}
Let $Q(x, y)=ax^2+bxy+cy^2$ be a primitive positive definite binary quadratic form with discriminant $-\Delta$ such that $8\nmid\Delta$ and $\Delta$ is not divisible by any odd prime squares. Then for any $A>0$ and any $\theta<1/2$, we have
\[
\sum_{q<x^\theta}\max_{\substack{u, v\imod q\\(Q(u, v), q)=1}}\max_{y\le x}\bigg|\sum_{\substack{Q(m, n)\le y\\m\equiv u\imod q\\n\equiv v\imod q}}\chip(Q(m, n))-\frac{\pi(y, Q)}{\nu(q)}\bigg|\ll_A x(\log x)^{-A}.
\]
where $\chip$ is the characteristic function for primes,
\[
\pi(x, Q)=\{(u, v)\in\mathbb{Z}^2, F(m, n)\le x\text{ and }Q(m, n)\text{ is prime.}\}
\]
and $\nu(q)$ is the number of $u, v\imod q$ such that $(Q(u, v), q)=1$.
\end{prop}
Applying Proposition \ref{BVbqf}, we are able to prove that
\begin{thm}
\label{Vat}
Let $Q(x, y)=ax^2+bxy+cy^2$ be a primitive positive definite binary quadratic form with discriminant $-\Delta$ such that $8\nmid\Delta$ and $\Delta$ is not divisible by any odd prime squares. Then there is a constant $c(Q)>0$ such that there are infinitely many primes of the form $p_1=Q(m, n)$ and $p_2=Q(m, n+h)$, with $m, n, h\in\Z$, such that $0<|h|\le c(Q)$.
\end{thm}

In Section 2, we will establish Proposition \ref{BVbqf} using the Bombieri-Vinogradov theorem for number fields developed by Huxley \cite{Hux}. The proof of Theorem \ref{Vat} will be provided in Section 3.\\

In the following sections we will fix a binary quadratic form $Q(x, y)=ax^2+bxy+cy^2$ and denote $-\Delta<0$ to be its discriminant. For simplicity we assume that $\Delta>4$. The proof for $\Delta\le4$ is similar but some changes are needed to account for the larger (but still finite) unit group.

\section{Bombieri-Vinogradov theorem for $Q(x, y)$}

Let $K$ be a number field with ring of integers $\mathcal{O}_K$. Let $r_1$ be the number of real embeddings of $K$ in $\mathbb{C}$, $h_K$ be the class number of $K$ and $U$ be the group of units in $\mathcal{O}_K$. Eventually we will take $K=\mathbb{Q}(-\Delta)$, and since $\Delta>4$ we have $U=\{\pm1\}$.\\

To generalize the notion of arithmetic progression, we introduce the notion of ray class group. For an integral ideal $\mathfrak{q}$ of $\mathcal{O}_K$, we consider the group of fractional ideals coprime to $\mathfrak{q}$ and define an equivalence relation between them: we say that $\mathfrak{a}\sim\mathfrak{b}$ if there exists $\alpha, \beta\in\mathcal{O}_K$ that satisfies:
\begin{enumerate}
\item $(\alpha), (\beta)$ coprime to $\mathfrak{q}$,
\item $\alpha-\beta\in\mathfrak{q}$,
\item $\alpha/\beta\succ0$, that is, $\sigma(\alpha/\beta)>0$ for all real embeddings $\sigma$ of $K$,
\end{enumerate}
such that
\[
(\alpha)\mathfrak{a}=(\beta)\mathfrak{b}.
\]
The equivalence classes form a group, the ray class group $\imod{\mathfrak{q}}$ and we denote it by $C_\mathfrak{q}$. Let $h(\mathfrak{q})$ denote the cardinality of $C_\mathfrak{q}$. Then the value of $h(\mathfrak{q})$ is provided by the formula
\[
h(\mathfrak{q})=\frac{2^{r_1}\phi(\mathfrak{q})h_K}{[U:U_{\mathfrak{q}, 1}]}
\]
where
\[
U_{\mathfrak{q}, 1}=\{\alpha\in U:\alpha\equiv1\imod{\mathfrak{q}}, \alpha\succ0\}
\]
and $\phi(\mathfrak{q})$ is defined by
\[
\phi(\mathfrak{q})=N(\mathfrak{q})\prod_{\mathfrak{p}|\mathfrak{q}}\bigg(1-\frac{1}{N(\mathfrak{p})}\bigg).
\]
For the proof see \cite{Milne}. Furthermore, the proof there also implies that $\phi(\mathfrak{q})$ is the number of elements of $(\mathcal{O}_K/\mathfrak{q})^*$. In this setting, Huxley \cite{Hux} proved the following generalization of Bombieri-Vinogradov theorem for number fields:
\begin{prop}[Huxley]
Let $\chi_{\mathfrak{P}}$ be the characteristic function for prime ideals and $\pi(x, K)$ be the number of prime ideals with norm less than $x$. Then for any $A>0$ and any $\theta<1/2$, we have
\[
\sum_{N(\mathfrak{q})\le x^\theta}\frac{h(\mathfrak{q})}{\phi(\mathfrak{q})}\max_{\mathfrak{C}_\mathfrak{q}\in C_\mathfrak{q}}\max_{y\le x}\bigg|\sum_{\substack{N(\mathfrak{a})\le y\\\mathfrak{a}\in\mathfrak{C}_\mathfrak{q}}}\chi_{\mathfrak{P}}(\mathfrak{a})-\frac{\pi(y, K)}{h(\mathfrak{q})}\bigg|\ll_A x(\log x)^{-A}.
\]
\end{prop}
Since we are only concerned with imaginary quadratic fields, we have $r_1=0$ and $\alpha\succ0$ for all $\alpha\in\mathcal{O}_K$. Since $\Delta\neq3, 4$, we also know that $[U:U_{\mathfrak{q}, 1}]=2$ unless $2\in\mathfrak{q}$. This only happens when $N(\mathfrak{q})=2, 4$, and is therefore negligible.\\

In Proposition \ref{BVbqf}, note that if $Q(x, y)$ represents an integer $k$, then we can apply a $\text{SL}_2(\mathbb{Z})$ action on $Q(x, y)$ so that $Q(1, 0)=k$. Since $Q(x, y)$ is known to represent infinitely many primes, we assume $Q(1, 0)$ is a sufficiently large prime that is coprime to all $q$ in Proposition \ref{BVbqf}. We also assume $(Q(m, n), q)=1$ since we only count prime values of $Q(m, n)$. The main idea of the proof of Proposition \ref{BVbqf} is to relate $Q(m, n)$ to certain integral ideals. This can be achieved by using
\[
aQ(m, n)=N((\alpha_1m+\alpha_2n))
\]
where
\[
\alpha_1=a, ~\alpha_2=\frac{-b+\sqrt{b^2-4ac}}{2}.
\]
From our assumptions, the ideal $(\alpha_1m+\alpha_2n)$ is coprime to the ideal $(q)$. Define the ideal $\mathfrak{a}=(\alpha_1, \alpha_2)$. Then we have $N\mathfrak{a}=a$ and $(\alpha_1m+\alpha_2n)\subseteq\mathfrak{a}$. Thus we can write $Q(m, n)=N(\mathfrak{s})$ where
\[
\mathfrak{as}=(\alpha_1m+\alpha_2n).
\]
On the other hand, if $\mathfrak{s}$ is an integral ideal of $\mathcal{O}_K$ such that $\mathfrak{as}=(\alpha)$ is principal, we must have $\alpha\in(\alpha_1, \alpha_2)$ and hence $\mathfrak{as}=(\alpha_1m+\alpha_2n)$ for some $m, n\in\mathbb{Z}$. Therefore we can work on the ideals $(\alpha_1m+\alpha_2n)/\mathfrak{a}$ with $m\equiv u, n\equiv v\imod q$ instead. To detect this congruence condition, we resort to ray class group as we promised.
\begin{prop}
Let $Q(x, y)=ax^2+bxy+cy^2$ be a primitive positive definite binary quadratic form with discriminant $-\Delta$ such that $8\nmid\Delta$ and $\Delta$ is not divisible by any odd prime squares. Let $q\in\N$ with $(a, q)=1$ and $u, v\in\mathbb{Z}$ with $(Q(u, v), q)=1$. Then for any $m, n\in\mathbb{Z}$, $m\equiv\pm u, n\equiv \pm v\imod q$ if and only if $(\alpha_1m+\alpha_2n)\sim(\alpha_1u+\alpha_2v)$ in $C_{(q)}$.
\end{prop}
\begin{proof}
WLOG suppose $(m, n)\equiv(u, v)\imod q$. Then
\[
(\alpha_1m+\alpha_2n)-(\alpha_1u+\alpha_2v)=mq
\]
for some $m\in\mathcal{O}_K$. Therefore
\[
(\alpha_1m+\alpha_2n-mq)(\alpha_1m+\alpha_2n)=(\alpha_1m+\alpha_2n)(\alpha_1u+\alpha_2v).
\]
By taking $\alpha=\alpha_1m+\alpha_2n-mq$ and $\beta=\alpha_1m+\alpha_2n$, we obtain $(\alpha)(\alpha_1m+\alpha_2n)=(\beta)(\alpha_1u+\alpha_2v)$ and hence $(\alpha_1m+\alpha_2n)\sim(\alpha_1u+\alpha_2v)$ in $C_{(q)}$.\\

On the other hand, if $(\alpha_1m+\alpha_2n)\sim(\alpha_1u+\alpha_2v)$, then there exists $\alpha, \beta$ coprime to $(q)$ such that $\alpha-\beta\in(q)$ and
\[
\alpha(\alpha_1m+\alpha_2n)=\pm\beta(\alpha_1u+\alpha_2v).
\]
Since $\alpha-\beta=tq$ for some $t\in\mathcal{O}_K$, this gives
\[
tq(\alpha_1m+\alpha_2n)=-\beta\bigg(\alpha_1(m\pm u)+\alpha_2(n\pm v)\bigg).
\]
As $\beta$ is coprime to $(q)$ we deduce that
\[
(q)\bigg|\bigg(\alpha_1(m\pm u)+\alpha_2(n\pm v)\bigg).
\]
This implies $(\alpha_1(m\pm u)+\alpha_2(n\pm v))/q\in\mathcal{O}_K$, or alternatively
\[
\frac{2a(m\pm u)-b(n\pm v)}{2q}+\frac{n\pm v}{2q}\sqrt{-\Delta}\in\mathcal{O}_K.
\]
If $b$ is even, then $\Delta$ is also even and $4||\Delta$. From the imaginary part we have $q|n\pm v$. From the real part it forces $2q|2a(m\pm u)$, and hence $q|m\pm u$ as well. If $b$ is odd, then $\Delta$ is also odd and the imaginary part gives $q|n\pm v$ again. Note that
\[
\frac{-b(n\pm v)}{2q}, \frac{n\pm v}{2q}
\]
are both integers or half-integers. Therefore $2a(m\pm u)/2q\in\mathbb{Z}$ and the result follows again.
\end{proof}

The above proposition shows that the classes in the ray class group indeed generalize the notion of arithmetic progressions, but up to a unit. Moreover, the function $\nu(q)$ is related to the Euler-phi function for integral ideals:
\begin{prop}\label{nu}
For any $q\in\mathbb{N}$, we have
\[
\nu(q)=\phi((q)).
\]
In particular, if $p$ is a prime, then
\[
\nu(p)=\begin{cases}
(p-1)^2\hspace{5mm}\text{ if $p$ splits,}\\
p^2-1\hspace{5mm}\text{ if $p$ is inert,}\\
p^2-p\hspace{5mm}\text{ if $p$ ramifies.}
\end{cases}
\]
\end{prop}
\begin{proof}
Since $(Q(1, 0), q)=1$, we have $(Q(m, n), q)=1$ if and only if $\alpha_1m+\alpha_2n$ is a unit in $(\mathcal{O}_K/\mathfrak{q})^*$ with $\mathfrak{q}=(q)$. The result then follows from the fact that $\phi(\mathfrak{q})$ is the number of elements in $(\mathcal{O}_K/\mathfrak{q})^*$.
\end{proof}
\begin{proof}[Proof of Proposition \ref{BVbqf}]
\[
\begin{split}
\sum_{\substack{Q(m, n)\le y\\u\equiv u\imod q\\v\equiv v\imod q}}\chip(Q(m, n))&=\sum_{\substack{Q(m, n)\le y\\(\alpha_1m+\alpha_2n)\sim(\alpha_1u+\alpha_2v)}}\chip\bigg(\frac{N(\alpha_1m+\alpha_2n)}{a}\bigg)=2\sum_{\substack{N\mathfrak{s}\le y\\\mathfrak{s}\in\mathfrak{C}_{(q)}}}\chi_\mathfrak{P}(\mathfrak{s})
\end{split}
\]
where $\mathfrak{C}_{(q)}=\langle(\alpha_1u+\alpha_2v)\rangle\langle\mathfrak{a}\rangle^{-1}$ and $\langle\cdot\rangle$ is the equivalence class that the ideal belongs to. On the other hand
\[
\begin{split}
\frac{\pi(y, K)}{h((q))}=\frac{2\pi(y, K)}{\nu(q)h_K}.
\end{split}
\]
Therefore
\[
\begin{split}
\sum_{\substack{N(\mathfrak{a})\le y\\\mathfrak{a}\in\mathfrak{C}_{(q)}}}\chip(\mathfrak{a})-\frac{\pi(y, K)}{h((q))}=&~\frac{1}{2}\bigg(\sum_{\substack{Q(m, n)\le y\\u\equiv u\imod q\\v\equiv v\imod q}}\chip(Q(m, n))-\frac{\pi(x, Q)}{\nu(q)}\bigg)\\
&+\frac{1}{\nu(q)}\bigg(\pi(x, Q)-\frac{\pi(y, K)}{h_K}\bigg).
\end{split}
\]
By Proposition \ref{nu} and the fact that $1/\nu(p)=1/p^2+O(1/p^3)$, we deduce that
\[
\begin{split}
&\sum_{d<x^\theta}\max_{\substack{u, v\imod d\\(Q(u, v)=1)}}\max_{y\le x}\bigg|\sum_{\substack{Q(m, n)\le y\\m\equiv u\imod d\\n\equiv v\imod d}}\chip(Q(m, n))-\frac{\pi(y, Q)}{\nu(d)}\bigg|\\
\ll&\sum_{N(\mathfrak{q})\le x^\theta}\max_{\mathfrak{C}_\mathfrak{q}\in C_\mathfrak{q}}\max_{y\le x}\bigg|\sum_{\substack{N(\mathfrak{a})\le y\\\mathfrak{a}\in\mathfrak{C}_\mathfrak{q}}}\chip(\mathfrak{a})-\frac{\pi(y, K)}{h(\mathfrak{q})}\bigg|
+\bigg|\pi(x, Q)-\frac{\pi(y, K)}{h_K}\bigg|\bigg(\sum_{q\le x^\theta}\frac{1}{\nu(q)}\bigg)\\
\ll&~x(\log x)^{-A}.
\end{split}
\]
This completes the proof of Proposition 1.
\end{proof}

\section{Bounded gap between primes represented by $Q(x, y)$}
For $k$ pairs of integers $(g_1, h_1), (g_2, h_2), ..., (g_k, h_k)$, we call them \textbf{admissible} with respect to $Q(x, y)$ if for any prime $p$, there exists $m, n\in\mathbb{Z}$ such that
\[
\prod_{\ell=1}^kQ(m+g_\ell, n+h_\ell)\not\equiv0\imod p.
\]
For bounded gap problem, it is now standard to consider
\begin{equation}
\label{now}
\sum_{Q(m, n)\sim X}\bigg(\sum_{j=1}^k\chip(Q(m+g_j, n+h_j))-\rho\bigg)\bigg(\sum_{\substack{d_1, d_2, ..., d_k\in\N\\d_i|Q(m+g_i, n+h_i)\i}}\lambda(d_1, d_2, ..., d_k)\bigg)^2
\end{equation}
for some $\rho>0$ and suitable sieve weights $\lambda$. To construct the sieve weights, let $\mathcal{F}:[0, \infty)^k\rightarrow[0, 1]$ be a symmetric smooth function that is supported on
\[
\Delta_k:=\{(t_1, ..., t_k)\in[0, \infty)^k:t_1+...+t_k\le1\}.
\]
Then we can define our sieve weights $\lambda(d_1, ..., d_k)$ by
\[
\lambda(d_1, ..., d_k)=\mu(d_1)\cdots\mu(d_k)\mathcal{F}\bigg(\frac{\log d_1}{\log R}, ..., \frac{\log d_k}{\log R}\bigg).
\]
The value of $R$ will be chosen to be a power of $X$. From the support of $\mathcal{F}$ we can see that $d_1d_2\cdots d_k\le R$. Our goal here is to show that there exist $\rho>1$ and $k\ge2$ such that \eqref{now} is positive for sufficiently large $X$. However, some difficulties might arise if $Q(m+g_j, n+h_j)$ are not coprime to each other. Fortunately these potential common factors must be relatively small:
\begin{lem}
Let $Q(x, y)=ax^2+bxy+cy^2$. If $Q(m, n)\equiv Q(m+g, n+h)\equiv0\imod p$, then $p|aQ(g, h)Q(g, -h)$.
\end{lem}
\begin{proof}
If $p\nmid n(n+h)$, then we have $mn^{-1}\equiv (m+g)(n+h)^{-1}\imod p$ and hence $mh\equiv \pm ng\imod p$. Therefore
\[
Q(mh, nh)\equiv Q(\pm ng, nh)\equiv n^2Q(g, \pm h)\imod p.
\]
This implies $p|Q(g, h)$ or $p|Q(g, -h)$. On the other hand if $p|n$, we have $p|am^2$. Thus we either have $p|a$ or $p|m$ and
\[
0\equiv Q(m+g, n+h)\equiv Q(g, h)\imod p.
\]
\end{proof}
\noindent Therefore we can put a restriction $m\equiv r_1, n\equiv r_2\imod W$ where $W$ is the product of all primes less than $\log\log\log X$ and $r_1, r_2$ are integers such that
\[
Q(r_1+g_j, r_2+h_j)\not\equiv0\imod W
\]
for $j=1, 2, ..., k$. This guarantees that $Q(m+g_j, n+h_j)$ are mutually coprime. Define
\[
\begin{split}
S(X, \rho):=&\sum_{\substack{Q(m, n)\sim X\\m\equiv r_1\imod W\\n\equiv r_2\imod W}}\bigg(\sum_{j=1}^k\chip(Q(m+g_j, n+h_j))-\rho\bigg)\\
&\times\bigg(\sum_{\substack{d_1, d_2, ..., d_k\in\N\\d_i|Q(m+g_i, n+h_i)\i}}\lambda(d_1, d_2, ..., d_k)\bigg)^2.
\end{split}
\]
By interchanging the order of summations, we obtain
\[
\begin{split}
S(X, \rho)=&\sum_{\substack{d_1, ..., e_k\in\N\\(d_ie_i, W)=1\i}}\lambda(d_1, ..., d_k)\lambda(e_1, ..., e_k)\sum_{\substack{a_i, b_i\imod {[d_i, e_i]}\\Q(a_i, b_i)\equiv0\imod{[d_i, e_i]}\i}}\\
&\sum_{\substack{Q(m, n)\sim X\\m\equiv r_1\imod W\\n\equiv r_2\imod W\\m\equiv a_i\imod{[d_i, e_i]}\i\\n\equiv b_i\imod{[d_i, e_i]}\i}}\bigg(\sum_{j=1}^k\chip(Q(m+g_j, n+h_j))-\rho\bigg)
\end{split}
\]
Since the $k+1$ integers $W$ and $[d_i, e_i]$ for $i=1, 2, ..., k$ are mutually coprime, the congruence conditions on $m, n$ in the innermost sum can be treated as a single congruence restriction modulo $W\prod_i[d_i, e_i]$. Note that
\begin{equation}
\label{Qarith}
\sum_{\substack{Q(m, n)\le X\\m\equiv a\imod q\\n\equiv b\imod q}}1=\frac{1}{q^2}\frac{2\pi X}{\sqrt{-\Delta}}+O(1).
\end{equation}
By Proposition \ref{BVbqf} and \eqref{Qarith}, we deduce that
\[
S(X, \rho)=\frac{1}{\nu(W)}\frac{\delta_QX}{h(-\Delta)\log X}\cdot\bigg(\sum_{\ell=1}^kS_{2, \ell}\bigg)-\rho\cdot\frac{1}{W^2}\frac{2\pi X}{\sqrt{-\Delta}}\cdot S_1+o(R^2)
\]
where
\[
S_1:=\sum_{\substack{d_1, ..., e_k\in\N\\(d_i, e_j)=1\text{ for $i\neq j$}}}\lambda(d_1, ..., d_k)\lambda(e_1, ..., e_k)\prod_i\frac{\rho([d_i, e_i])}{[d_i, e_i]^2}
\]
and
\[
S_{2, \ell}:=\sum_{\substack{d_1, ..., e_k\in\N\\(d_i, e_j)=1\text{ for $i\neq j$}\\d_\ell=e_\ell=1}}\lambda(d_1, ..., d_k)\lambda(e_1, ..., e_k)\prod_i\frac{\rho([d_i, e_i])}{\nu([d_i, e_i])}.
\]
These two types of sums also appear in many other works on bounded gap between primes and from today's perspective they can be evaluated by fairly standard Fourier-analytic techniques. Here we simply quote the following proposition, which is a special case of Lemma 2.3.1 in \cite{Vatthesis}.
\begin{prop}
\label{Fourier}
Let $g$ be a multiplicative function with
\begin{equation}
\label{gg}
\frac{1}{g(p)}=p+O(p^t)
\end{equation}
with $t<1$. Define
\[
S_1(g)=\sum_{\substack{d_1, ..., e_k\in\N\\(d_i, e_j)=1\text{ for $i\neq j$}}}\lambda(d_1, ..., d_k)\lambda(e_1, ..., e_k)\prod_ig([d_i, e_i])
\]
and
\[
S_{2, \ell}(g)=\sum_{\substack{d_1, ..., e_k\in\N\\(d_i, e_j)=1\text{ for $i\neq j$}\\d_\ell=e_\ell=1}}\lambda(d_1, ..., d_k)\lambda(e_1, ..., e_k)\prod_ig([d_i, e_i]).
\]
Then we have
\[
S_1(g)=(1+o(1))\bigg(\frac{W}{\phi(W)}\bigg)^k\frac{I(\mathcal{F})}{(\log R)^k},
\]
and
\[
S_{2, \ell}(g)=(1+o(1))\bigg(\frac{W}{\phi(W)}\bigg)^k\frac{J_\ell(\mathcal{F})}{(\log R)^{k-1}}.
\]
Here
\[
I(\mathcal{F}):=\int_{\Delta_k}\bigg(\frac{\partial^k\mathcal{F}}{\partial t_1...\partial t_k}\bigg)^2\, dt_1\, dt_2\, ...\, dt_k
\]
and
\[
J_\ell(\mathcal{F}):=\int_{\Delta_{k-1}}\bigg(\frac{\partial^{k-1}\mathcal{F}}{\partial t_1...\partial t_{\ell-1}\partial t_{\ell+1}...\partial t_k}\bigg)^2(t_1, ..., t_{\ell-1}, 0, t_{\ell+1}, ..., t_k)\, dt_1\, dt_2\, ...\,dt_{\ell-1}\, dt_{\ell+1}\, ...\, dt_k.
\]
\end{prop}
Since $g(d)=\rho(d)/d^2$ and $g(d)=\rho(d)/\nu(d)$ both satisfy \eqref{gg}, by Proposition \ref{Fourier} we obtain
\[
S(X, \rho)=(1+o(1))\bigg(\frac{W}{\phi(W)}\bigg)^k\frac{X}{(\log R)^k}\bigg(\frac{1}{\nu(W)}\bigg(\frac{\theta}{2}-\delta\bigg)\frac{\delta_QkJ_\ell(\mathcal{F})}{h(-\Delta)}-\frac{\rho}{W^2}\frac{2\pi}{\sqrt{-\Delta}}I(\mathcal{F})\bigg).
\]
Finally since $W^2\ge\nu(W)$, it suffices to find $\rho>1$ such that
\begin{equation}
\label{rho}
\rho<\bigg(\frac{\theta}{2}-\delta\bigg)\frac{\delta_Q\sqrt{-\Delta}}{2\pi h(-\Delta)}\frac{kJ_k(\mathcal{F})}{I(\mathcal{F})}.
\end{equation}
By Theorem 23 of \cite{Poly}, for sufficiently large $k$, we have
\[
\sup_{\mathcal{F}}\frac{kJ_k(\mathcal{F})}{I(\mathcal{F})}\gg\log k.
\]
Thus when $k$ is large enough, the right side of \eqref{rho} is strictly greater than 1. Therefore we can pick $\rho>1$ and this proves Theorem \ref{Vat}.

\bibliographystyle{amsbracket}
\providecommand{\bysame}{\leavevmode\hbox to3em{\hrulefill}\thinspace}

\end{document}